\documentclass[a4paper]{article}
\usepackage[utf8x]{inputenc}
\usepackage[english]{babel} 
\usepackage{amsmath,amssymb} 
\usepackage{calrsfs}
\usepackage{graphicx}
\usepackage{esint}
\usepackage{booktabs}
\usepackage{chngpage}
\usepackage{todonotes}
\usepackage{graphicx}
\usepackage{enumerate}
\usepackage[margin=2cm]{geometry}
\usepackage{epstopdf}
\usepackage{listings}
\usepackage{float, tikz, tikz-3dplot, tikz-cd, framed}
\usetikzlibrary{calc}
\usepackage{hyperref}
\hypersetup{colorlinks={true},linkcolor={black}}
\DeclareMathAlphabet{\pazocal}{OMS}{zplm}{m}{n}

\newtheorem{theorem}{Theorem}

\newtheorem{corollary}[theorem]{Corollary}

\newtheorem{lemma}[theorem]{Lemma}

\newtheorem{remark}[theorem]{Remark}

\numberwithin{theorem}{section}

\newenvironment{proof}[1][Proof]{\textbf{#1.} }{\ \rule{0.5em}{0.5em}}

\makeatletter
 {\par\unskip\endMakeFramed}
\makeatother

\renewcommand{\d}[1]{\ensuremath{\operatorname{d}\!{#1}}}

\newsavebox{\overlongequation}

\begin{document}
\title{Lyapunov Exponents of Two Stochastic Lorenz 63 Systems
\\ \bigskip\Large
Bernard J. Geurts$^1$, Darryl D. Holm$^2$ and Erwin Luesink$^{2}$ 
\\ \bigskip\small
$^1$ Applied Mathematics, University of Twente, Enschede 7500 AE, NL\\
$^2$ Mathematics, Imperial College London SW7 2AZ, UK}
\date{}                                           

\maketitle

\makeatother

\begin{abstract}
Two different types of perturbations of the Lorenz 63 dynamical system for Rayleigh-B\'enard convection by multiplicative noise -- called stochastic advection by Lie transport (SALT) noise and fluctuation-dissipation (FD) noise -- are found to produce qualitatively different effects, possibly because the total phase-space volume contraction rates are different. In the process of making this comparison between effects of SALT and FD noise on the Lorenz 63 system, a stochastic version of a robust deterministic numerical algorithm for obtaining the individual numerical Lyapunov exponents was developed. With this stochastic version of the algorithm, the value of the sum of the Lyapunov exponents for the FD noise was found to differ significantly  from the value of the deterministic Lorenz 63 system, whereas the SALT noise preserves the Lorenz 63 value with high accuracy. The Lagrangian averaged version of the SALT equations (LA SALT) is found to yield a closed deterministic subsystem for the expected solutions which is found to be isomorphic to the original Lorenz 63 dynamical system. The solutions of the closed chaotic subsystem, in turn, drive a linear stochastic system for the fluctuations of the LA SALT solutions around their expected values.   
\end{abstract}

\section{Introduction}
A fundamental challenge in geophysical fluid dynamics (GFD) is the estimation of combined measurement error and model uncertainty. These errors can arise, for example, from unknown or neglected physical scales, incomplete information in the data and incomplete formulation of the theoretical models. The modern methodology of stochastic data-driven modelling has been developed for managing the loss of predictability  associated with error and uncertainty. A fundamental tenet of our approach to this methodology is that stochastic data-driven models should respect the physical principles underlying their deterministic counterparts. For example, if the deterministic physical model follows from Hamilton's variational principle, one may introduce stochastic data dependence as a constraint on the deterministic variational principle. Recently, \cite{holm2015variational} followed this approach for fluid dynamics by introducing Stochastic Advection by Lie Transport (SALT) as a constraint on Hamilton's principle for ideal fluid dynamics. The SALT variational structure preserves both the physical conservation laws and their mathematical implications, such as the Kelvin circulation theorem. Thus, data-driven stochastic fluid flow models based on the SALT approach will obey the same laws of physics as their corresponding deterministic models, although the time dependence for advection by fluid flow will become stochastic.
\bigskip

Now, the solutions of all GFD models of climate and weather tend to be chaotic in nature. That is, some aspects of their solutions tend to show sensitive dependence on initial conditions and on time dependence of forcing. For dynamical systems, this sensitive dynamical dependence is exhibited by positive Lyapunov exponents, meaning that initially nearby solutions diverge away from each other exponentially in time. In addition, GFD models tend to admit bifurcations of solution behaviour depending on their model parameters. Thus, for GFD one must take on the challenge of modelling measurement error and uncertainty which is also nonlinearly concatenated with chaotic deterministic dynamics, is therefore sensitive to initial conditions as well as forcing, and is prone to bifurcations. The most famous example of chaos in GFD is the chaotic dynamics on the strange attractor of the deterministic Lorenz 63 system \cite{lorenz1963deterministic}, which may be represented for $(X,Y,Z)^T\in\mathbb{R}^3$ as the dynamical system,
\begin{equation}
\begin{aligned}
\frac{d}{d\tau}X &= \sigma(Y-X),\\
\frac{d}{d\tau}Y &= (rX-XZ-Y),\\
\frac{d}{d\tau}Z &= (XY - bZ),
\end{aligned}
\label{Lorenz63}
\end{equation}
where $\sigma$, $r$, $b$, are real constants and $\tau$ is time. The physical meaning of the Lorenz 63 system \eqref{Lorenz63} and some of its main properties will be recalled below. The paper begins with a review of its derivation from the Rayleigh-B\'enard model of thermal convection of fluids heated from below which includes an explanation of how its deterministic derivation accommodates the transition to Stochastic Advection by Lie Transport (SALT).
\bigskip

{\bf Brief summary.} This paper aims to incorporate SALT into the Rayleigh-B\'enard model of thermal convection and then follow the standard historical approach to derive the stochastic Lorenz 63 system augmented by SALT. It also compares the dynamical simulation results for the numerical Lyapunov exponents (NLEs) of the SALT Lorenz 63 model with those of the stochastic Lorenz 63 system investigated in \cite{chekroun2011stochastic}.  
The individual NLEs of the two cases appear to be almost identical for each realisation of the noise. However, the sums are different, so the total phase-space volume contraction rates are different. Namely, the SALT approach preserves the phase-space contraction rate of the underlying deterministic Lorenz 63 system, whereas the stochastic Lorenz 63 system in \cite{chekroun2011stochastic} varies with each realisation. The implications of this difference in temporal behaviour of total contraction rates are elicited in numerical simulations. After a brief Conclusion section which summarises the main resutls, the paper's Outlook section briefly discusses a toy climate model based on the SALT Lorenz 63 system. In particular, the relations between `what you expect' and the dynamical evolution of the statistics in `what you get' are discussed in the climate context. 

\subsubsection*{Background}
In \cite{holm2015variational}, SALT was introduced by using Hamilton's variational principle for fluids, constrained to enforce stochastic Lagrangian fluid trajectories, arising as characteristic curves of the following stochastic Eulerian velocity vector field
\begin{equation}
v(x,t,dW) := u(x,t)dt + \sum_{i=1}^N \xi_i(x)\circ dW^i \,.
\label{decomposition}
\end{equation} 
The velocity vector field \eqref{decomposition} is regarded as a decomposition into a deterministic drift velocity $u(x,t)$ and a sum over spatially dependent but temporally constant vector field amplitudes of Stratonovich stochastic terms. The $\xi_i(x)$ are taken as stationary eigenvectors of the velocity-velocity correlation tensor. In practice, the $\xi_i(x)$ need to be supplied a priori. See \cite{cotter2018modelling, cotter2019numerically, cotter2019particle} for how the $\xi_i(x)$ are determined in practical applications. Stratonovich stochastic calculus is preferred in the derivation of the SALT equations, because its standard calculus properties (product rule and chain rule) also admit variational calculus. The decomposition \eqref{decomposition} was shown in \cite{cotter2017stochastic} to also arise from homogenisation of the velocity obtained by writing the deterministic time-dependent Lagrange-to-Euler diffeomorphism as the composition of two diffeomorphisms, a fluctuating map with two time scales $(t,t/\epsilon)$ where $\epsilon\ll 1$, applied to a mean flow map depending only on the slow time scale $t$.  In this analysis, the $\xi_i(x)$ should be understood as empirical orthogonal functions which correspond to the different modes of the fast flow. The rigorous application of homogenisation theory requires added assumptions of mildly chaotic fast small-scale deterministic dynamics, as well as a centering condition, according to which the mean of the fluctuating deviations is small, when pulled back to the mean flow. The homogenisation analysis performed in \cite{cotter2017stochastic} gives rise to It\^o noise, which is then transformed to the Stratonovich type.
\bigskip

The work of \cite{cotter2017stochastic} justified regarding the Eulerian vector field in \eqref{decomposition} as a genuine decomposition of the fluid velocity into a sum of drift and stochastic parts, rather than simply a perturbation of the dynamics meant to model unknown effects in uncertainty quantification. Consequently, one should expect that the properties of the fluid equations with SALT should closely track the properties of the unapproximated solutions of the fluid equations. For example, if the unapproximated model equations are Hamiltonian, then the model equations with SALT should also be Hamiltonian. This was shown in \cite{holm2015variational}. 
\bigskip

In \cite{cotter2018modelling} and \cite{cotter2019numerically}, respectively, a 2-layer quasi-geostrophic model and the 2D Euler equations were each perturbed using the SALT velocity decomposition in \eqref{decomposition}. These investigations were aimed at developing a new methodology for uncertainty quantification and data assimilation for ideal fluids. The new methodology tested in these papers was found to be suitable for coarse graining in both cases. Specifically, uncertainty quantification tests of the velocity decomposition of \cite{cotter2017stochastic} were performed by comparing ensembles of coarse-grid realisations of solutions of the resulting stochastic partial differential equations with the ``true solutions'' of the deterministic fluid partial differential equations obtained by computing the same problem at higher resolution. The time discretisations used for approximating the solution of the stochastic partial differential equations were shown to be consistent in each case and comprehensive numerical tests confirmed the non-Gaussianity of the flows under SALT dynamics.
\bigskip

In \cite{crisan2019solution} it is shown that the Euler equations for an ideal, inviscid fluid with SALT are locally well-posed in regular spaces and a Beale-Kato-Majda type blow-up criterion is proved. Thus the analytical properties of the 3D Euler fluid equations with SALT closely mimic the corresponding analytical properties of the original deterministic 3D Euler equations.
\bigskip

Inspired by spatiotemporal observations from satellites of the trajectories of objects drifting near the surface of the ocean in the National Oceanic and Atmospheric Administration's (NOAA) Global Drifter program, in \cite{gay2018stochastic} a data-driven stochastic model of geophysical fluid dynamics was developed. Here non-stationary spatial correlations represent the effects of advection by ocean currents. The methods in this paper are similar to those in \cite{holm2015variational}, where the models were derived using reduction by symmetry of stochastic variational principles. The corresponding momentum maps, conservation laws and Lie-Poisson bracket structures were used in developing the new stochastic Hamiltonian models of geophysical fluid dynamics with nonlinearly evolving stochastic properties. 

\bigskip

For motivation of the present work, we recall that the complex bifurcation behaviour shown in the Rayleigh-B\'enard model of thermally driven convection is quite challenging to simulate numerically, as discussed, e.g., in  \cite{kooij2018comparison, plumley2016effects, ostilla2013optimal, shishkina2007fourth}. One might hope that the SALT approach could play a useful role for investigating this convection model further. In particular, the large amount of available experimental and numerical data for thermally driven convection could, in principle, facilitate determination of the necessary correlation eigenvectors $\xi_i(x)$ in the decomposition of the fluid velocity  vector field above in equation \eqref{decomposition}. One would expect progress in this matter to produce stochastic models which are numerically much  less expensive than the corresponding deterministic models. In this paper, we shall derive the SALT modification of  the Rayleigh-B\'enard convection equations and then derive the corresponding Lorenz 63 model. We hope this study will stimulate future research in using SALT to develop efficient stochastic numerical simulations of Rayleigh-B\'enard convection. For example, the effects of stochasticity implemented by the SALT approach on the bifurcation structure of Rayleigh-B\'enard convection could be an avenue for further investigations of interest in fluid dynamics for both theory and applications. However, such an investigation is beyond the scope of the present paper. 

\subsubsection*{Plan of the paper}
In section 2 of this paper, SALT is introduced into the formulation of the Rayleigh-B\'enard convection, by using the methods from geometric mechanics \cite{holm2015variational}. In particular, the Kelvin circulation theorem is used to introduce SALT into the Oberbeck-Boussinesq (OB) equations which govern the Rayleigh-B\'enard model. In section 3, the OB equations with SALT are restricted to a vertical slice and by means of a truncated spatial Fourier series expansion are projected onto low wavenumber modes, following the approach of \cite{lorenz1963deterministic}. This process gives rise to a stochastic version of the Lorenz 63 system. In section 4, we compare this stochastic Lorenz 63 model derived using the SALT approach with an alternative stochastic Lorenz 63 model found in the literature \cite{chekroun2011stochastic}, in which linear multiplicative noise is added into the equation for each variable in \eqref{Lorenz63}. By using the theory of Lyapunov exponents and Oseledet's multiplicative ergodic theorem, analytical statements can be derived to describe the average rate of contraction of phase-space volume. In section 5, a numerical method to compute the numerical Lyapunov exponents (NLEs) is introduced, based on a QR method involving the Cayley-transform, as shown in \cite{udwadia2002efficient}. This method is generalised to deal with systems of stochastic differential equations and is shown to give accurate and robust results. Section  \ref{sec: conclude} summarises our conclusions. Section \ref{sec: outlook} briefly discusses an intriguing outlook for modifications of the present approach which could conceivably apply in climate science.

\section{Rayleigh-B\'enard convection and Lorenz 63}
In this paper, the goal is to look at the effect of SALT on the Lorenz 63 system in \eqref{Lorenz63}. The Lorenz 63 system can be systematically derived by expanding the equations of Rayleigh-B\'enard convection into Fourier series and then truncating. The Rayleigh-B\'enard model may be interpreted as a simple, local, weather system, where only heat and wind are involved. One of the important ingredients that is missing from the model is the effect of rotation, that gives rise to the Coriolis force and hence we can only think of this as a local model for weather effects. The domain is a three dimensional box, where the bottom plate is being heated and the top plated is being cooled at a constant temperature. This corresponds to the earth heating the surface of the earth, which heats the air directly above it. The hot air has a lower density than the cool air which is above it. Gravity acts to restore stability, causing the hot air to rise and the cool air to descend. The top plate is being cooled, corresponding to the cold air above the domain which acts as a heat sink. The fluid in the box is modelled with the Navier-Stokes equations under the Oberbeck-Boussinesq approximation. The fluid is assumed to have a constant heat capacity, so the advection-diffusion equation for the heat can be expressed in terms of temperature. The governing set of equations are
\begin{equation}
\begin{aligned}
\frac{\partial}{\partial t}\textbf{u} + \textbf{u}\cdot\nabla \textbf{u} &= -\nabla p + \nu \Delta \textbf{u} + \textbf{F} ,\\
\frac{\partial}{\partial t} T + \textbf{u}\cdot \nabla T &= \gamma\Delta T,\\
\nabla\cdot \textbf{u} &= 0.
\end{aligned}
\label{DOBeqns}
\end{equation}
The buoyancy force $\textbf{F} = \alpha g T \hat{\textbf{z}}$ in the top equation in \eqref{DOBeqns} acts in the vertical direction and depends on the thermal expansion coefficient $\alpha$, gravity $g$ and the local temperature $T$. This model is dissipative in nature due to the linear diffusion of momentum per unit mass $\textbf{u}$, by viscosity $\nu$, and heat per unit mass $T$, by heat diffusivity $\gamma$. The dissipation makes the dynamics irreversible, which does not fit the mathematical framework that is necessary to introduce SALT. At this stage, let us drop the dissipative terms temporarily, to illustrate the framework. Suppose $T$ is the heat per unit mass of a Lagrangian fluid parcel following the flow given by the smooth, invertible, time-dependent Lagrange-to-Euler map,
\begin{equation}
\eta_t X := \eta(X,t)\in\mathbb{R}^3, \qquad \text{for initial position}\quad x(X,0) = \eta_0 X = X.
\label{LagrangetoEulermap}
\end{equation}
By using the pullback operation, denoted by $\eta_t^*$, the temperature equation can be rewritten in terms of the Lagrange-to-Euler map \eqref{LagrangetoEulermap}. In particular, for the parcel occupying spatial position $x\in\mathbb{R}^3$ at time $t$, we have $\eta_t^* T(X,t) = T(\eta(X,t),t)$. The Eulerian velocity vector along a Lagrangian path may be written in terms of the flow $\eta_t$ and its pullback $\eta_t^*$ as
\begin{equation}
\frac{d\eta^j(X,t)}{dt} = u^j(\eta(X,t),t) = \eta^*_t u^j(X,t) = u^j(\eta_t^* X,t).
\label{pullbackTangent}
\end{equation}
The time derivative of the pullback of $\eta_t$ for the scalar function $T(X,t)$ is given by the chain rule as
\begin{equation}
\frac{d}{dt}\eta_t^* T(X,t) = \frac{d}{dt} T(\eta_t X,t) = \frac{\partial}{\partial t} T(\eta(X,t),t) + \frac{\partial T}{\partial \eta^j}\frac{\partial \eta^j(X,t)}{\partial t} = \eta^*_t\left(\frac{\partial}{\partial t}T + \frac{\partial T}{\partial x^j}u^j\right)(X,t).
\label{pullbackTemperature}
\end{equation}
Equation \eqref{pullbackTemperature} shows the geometric nature of advection of heat per unit mass along a Lagrangian particle path. A similar relation can be shown to exist for the momentum per unit mass $\textbf{u}$, although this relation involves an additional term, as we shall see below in the proof of Theorem \ref{Thm2.2}.  
\begin{remark}[Distinguishing between flow velocity and momentum per unit mass]
In the momentum equation in \eqref{DOBeqns}, two quantities with the dimensions of velocity appear, both noted as $\textbf{u}$. In Cartesian coordinates on $\mathbb{R}^n$ equipped with the Euclidean metric, when the kinetic energy is given by the $L^2$ metric, this is valid. In different metrics or different spaces, one has to distinguish between the contravariant vector field $u^j\partial_j$, which transports fluid properties and the covariant momentum per unit mass $u_j dx^j$, which is transported. To emphasize this difference, we will henceforth denote the vector field, or transport velocity, by $\widetilde{\textbf{u}}$ and write momentum per unit mass by $\textbf{u}$. Then equation \eqref{pullbackTangent} is written as
\begin{equation}
\frac{d\eta^j(X,t)}{dt} = \widetilde{u}^j(\eta(X,t),t) := \eta^*_t \widetilde{u}^j(X,t) = \widetilde{u}^j(\eta_t^* X,t),
\label{pullbackTangent2}
\end{equation}
and we can write the ideal Oberbeck-Boussinesq equations for the Rayleigh-B\'enard model as
\begin{equation}
\begin{aligned}
\frac{\partial}{\partial t}\textbf{u} + \widetilde{\textbf{u}}\cdot\nabla\textbf{u} &= -\nabla p + \textbf{F},\\
\frac{\partial}{\partial t}T + \widetilde{\textbf{u}}\cdot\nabla T &= 0,\\
\nabla\cdot\widetilde{\textbf{u}} &= 0.
\end{aligned}
\label{IOBeqns}
\end{equation}
The distinction between $\widetilde{\textbf{u}}$ and $\textbf{u}$ is particularly important in the Kelvin circulation theorem, where $\widetilde{\textbf{u}}$ transports the fluid loop and $\textbf{u}$ will be the circulation velocity, which is integrated around the loop.
\end{remark}

\begin{theorem}[Kelvin circulation theorem for the ideal Oberbeck-Boussinesq equations]\label{Thm2.2}
For the circulation integral
\begin{align*}
I(t) = \oint_{c(t)} \textbf{u}\cdot d\textbf{x},
\end{align*}
where $c(t)$ is a closed loop that is moving with velocity $\widetilde{\textbf{u}}$, the ideal Oberbeck-Boussinesq equations given in \eqref{IOBeqns} imply the following circulation dynamics,
\begin{align*}
\frac{d}{dt}I(t) = \oint_{c(t)} \alpha g T dz
\end{align*}
\end{theorem}
\begin{proof}
\begin{align*}
\frac{d}{dt}I(t) &= \frac{d}{dt} \oint_{c(t)} \textbf{u}\cdot d\textbf{x} = \oint_{c(0)}\frac{d}{dt}\eta_t^*(\textbf{u}\cdot d\textbf{x})\\
&=  \oint_{c(0)}\eta^*_t\left(\left(\frac{\partial}{\partial t}\textbf{u} + \widetilde{\textbf{u}}\cdot\nabla\textbf{u} + u_j\nabla \widetilde{u}^j\right)\cdot d\textbf{x}\right)\\
&=  \oint_{c(t)}\left(\frac{\partial}{\partial t}\textbf{u} + \widetilde{\textbf{u}}\cdot\nabla\textbf{u} + u_j\nabla \widetilde{u}^j\right)\cdot d\textbf{x}\\
&= \oint_{c(t)}\left(-\nabla p - u_j\nabla\widetilde{u}^j + \alpha g T\nabla z\right)\cdot d \textbf{x}
\end{align*}
Since the space that we are working in is $\mathbb{R}^3$, with Cartesian coordinates, an Euclidean metric and the kinetic energy is given by the $L^2$ metric, $\textbf{u}_j$ and $\textbf{u}^j$ can be identified as vectors. The identity $u_j\nabla u^j = \frac{1}{2}\nabla|\textbf{u}|^2$ is then valid and can be used to find
\begin{align*}
\frac{d}{dt}I(t) &= \oint_{c(t)}\left(-\nabla\left( p + \frac{1}{2}|\textbf{u}|^2\right) + \alpha g T\nabla z\right)\cdot d\textbf{x}\\
&= \oint_{c(t)}-d\left( p + \frac{1}{2}|\textbf{u}|^2\right) + \left(\alpha g T\right)dz\\
&=\oint_{c(t)} \alpha g T dz\,.
\end{align*}
In the last step we have used the fundamental theorem of calculus.
\end{proof}
\bigskip

The dissipative Oberbeck-Boussinesq equations \eqref{DOBeqns} possess the following circulation dynamics.
\begin{corollary}[Kelvin circulation theorem for the dissipative Oberbeck-Boussinesq equations]
The dissipative Oberbeck-Boussinesq equations \eqref{DOBeqns} possess the following circulation dynamics,
\begin{align}
\frac{d}{dt}I(t) = \oint_{c(t)} (\alpha g T\nabla z + \nu\Delta\textbf{u})\cdot d\textbf{x}\,,
\label{KCTheorem}
\end{align}
for the circulation integral given by
\begin{align*}
I(t) = \oint_{c(t)} \textbf{u}\cdot d\textbf{x}\,,
\end{align*}
where $c(t)$ is a closed loop moving with transport velocity $\widetilde{\textbf{u}}$.
\end{corollary}
\begin{remark}
The Kelvin theorem \eqref{KCTheorem} is secretly Newton's second law $\frac{d\mathbf{P}}{dt} = \mathbf{F}$ for the momentum per unit mass $\mathbf{P}$ and the force per unit mass $\mathbf{F}$ for masses distributed on a space of loops. From this viewpoint, the circulation is momentum per unit mass and the right hand side of \eqref{KCTheorem} is the force per unit mass.
\end{remark}
We now replace the velocity $\widetilde{\textbf{u}}$ that transports the closed loop $c(t)$ with a stochastic process for the Lagrangian trajectory given by
\begin{align}
\widehat{\textbf{u}}\rightarrow {\sf \textcolor{red} d}\textbf{y}_t = \textbf{u}\,dt + \sum_i {\boldsymbol \xi}_i\circ dW_t^i\,,
\label{SALTkelvinthm}
\end{align}
in which the assumption is made that $\nabla\cdot {\boldsymbol \xi}_i = 0$ for all $i = 1,\hdots, N$. 
This replacement introduces stochasticity into the dissipative Oberbeck-Boussinesq equations \eqref{DOBeqns} as follows
\begin{equation}
\begin{aligned}
{\sf \textcolor{red} d}\textbf{u}+ {\sf \textcolor{red} d}\textbf{y}_t\cdot\nabla\textbf{u} + u_j\nabla{\sf \textcolor{red} d}y_t^j &= (-\nabla p + \alpha g T \nabla z + \nu\Delta\textbf{u}) dt,\\
{\sf \textcolor{red} d}T + {\sf \textcolor{red} d}\textbf{y}_t\cdot\nabla T &= \gamma \Delta T dt,\\
\nabla\cdot {\sf \textcolor{red} d}\textbf{y}_t &= 0, 
\end{aligned}
\label{SOB}
\end{equation}
\bigskip

In stochastic fluid dynamics with SALT, the advective velocity field transforms, as we have seen. The framework for this transformation can be found in \cite{holm2015variational}. The Eulerian velocity field becomes the stochastic vector field ${\sf \textcolor{red} d}\textbf{y}_t$, which also advects the loop in the Kelvin circulation theorem. Here, this change was applied to the Oberbeck-Boussinesq equations to derive their stochastic analogue with SALT. The development of model equations for stochastic fluid dynamics revolves around the choice of forces appearing in Newton's second law and Kelvin's circulation theorem. A similar approach can be found in \cite{memin2014fluid,resseguier2017mixing}, where the goal is stochastic turbulence modelling. In \cite{chapron2018large} a stochastic version of Lorenz 63 can be found with a similar structure, as the noise was introduced carefully in the transport terms.

\subsection{Lorenz 63 equations}
In order to be able to project the equations onto the Fourier modes used in \cite{lorenz1963deterministic}, the OB equations in \eqref{SOB} need to written in terms of scalar vorticity and temperature profile, as in \cite{saltzman1962finite}. The temperature $T(x,z,t)$ at constant $y$ can be expanded into a horizontal mean value and a departure from that mean
\begin{equation}
T(x,z,t) = \overline{T}(z,t) + T'(x,z,t),
\end{equation}
where $\overline{T}$ is the horizontal mean and $T'$ is the departure from the mean. Also the mean temperature can be decomposed into a linear difference between the lower and upper boundary and a perturbation of this linear difference. This gives
\begin{equation}
\overline{T}(z,t) = \overline{T}(0,t) - \frac{T_\Delta}{H} z + \overline{T}''(z,t),
\end{equation}
where $T_{\Delta}$ is the constant temperature difference between the lower and upper boundary, $H$ is the height between the vertical boundaries and $\overline{T}''$ is the perturbation of the linear difference. The temperature can therefore be written as
\begin{equation}
T(x,z,t) = \left( \overline{T}(0,t) - \frac{T_\Delta}{H}z\right) + T'(x,z,t) +  \overline{T}''(z,t).
\label{tempexp}
\end{equation}
The temperature profile is defined as 
\begin{equation}
\phi(x,z,t) := T'(x,z,t) + \overline{T}''(z,t).
\end{equation}
Substitution of equation \eqref{tempexp} into the convection-diffusion equation for the heat yields
\begin{equation}
{\rm \textcolor{red}{d}} \phi + {\sf \textcolor{red} d}\textbf{y}_t\cdot\nabla \phi = \left(\frac{T_\Delta}{H}w + \gamma \Delta \phi\right) dt,
\label{tempprof}
\end{equation}
where $w$ is the vertical component of the velocity field. The vorticity equation is obtained by taking the curl of the momentum equation in \eqref{SOB}. This eliminates the pressure term and since the domain is two dimensional, the vorticity is a scalar function. This gives
\begin{equation}
{\rm \textcolor{red}{d}} \omega + {\sf \textcolor{red} d}\textbf{y}_t\cdot\nabla \omega = \left(\nu\Delta\omega + \alpha g \phi_x\right)dt.
\label{vorticity}
\end{equation}
The velocity terms in \eqref{tempprof} and \eqref{vorticity} can be rewritten in terms of streamfunctions. The streamfunction associated to the deterministic velocity field $\textbf{u}$ will be denoted by $\psi$ and the streamfunction associated to the stochastic vector field ${\sf \textcolor{red} d}\textbf{y}_t$ will be denoted by $\widetilde{\psi}$,
\begin{equation}
\begin{aligned}
{\rm \textcolor{red}{d}} \omega + \left|\frac{\partial(\widetilde{\psi},\omega)}{\partial(x,z)}\right| &= \left(\nu\Delta\omega + \alpha g \phi_x\right)dt,\\
{\rm \textcolor{red}{d}} \phi + \left|\frac{\partial(\widetilde{\psi},\phi)}{\partial(x,z)}\right|  &= \left(\gamma \Delta \phi - \frac{T_\Delta}{H}\psi_x\right) dt,\\
\omega &= -\Delta \psi.
\end{aligned}
\label{Lorenzformulation}
\end{equation}
This is the same formulation of the Rayleigh-B\'enard convection problem as in \cite{saltzman1962finite, lorenz1963deterministic}. The nonlinear terms have been rewritten in terms of a determinant of a Jacobian. Lorenz used the following truncated Fourier series
\begin{equation}
\begin{aligned}
\frac{k}{\gamma(1+k^2)}\psi &= X\sqrt{2}\sin\left(\frac{k\pi x}{H}\right)\sin\left(\frac{\pi z}{H}\right),\\
\frac{\pi R_a T_\Delta}{R_c}\phi &= Y\sqrt{2}\cos\left(\frac{k\pi x}{H}\right)\sin\left(\frac{\pi z}{H}\right) - Z\sin\left(\frac{2\pi z}{H}\right).
\end{aligned}
\label{deterministicfouriermodes}
\end{equation} 
In these Fourier expansions, $k$ is the wave number, $R_a = \alpha g H^3 T_\Delta \nu^{-1}\gamma^{-1}$ is the Rayleigh number and $R_c = \pi^4 k^{-2}(1+k^2)^3$ is the critical value of the Rayleigh number. These scaling constants are introduced so that the resulting equations take a compact form. We take the truncated Fourier series of the stochastic streamfunction $\widetilde{\psi}$ to be the same as the expansion of the deterministic streamfunction $\psi$, but with a stochastic coefficient. The motivation for this choice is that from a physical point of view, we do not want the stochasticity to give rise to types of motion other than rolls between the two plates. The expansion of the stochastic streamfunction then is
\begin{equation}
\frac{k}{\gamma(1+k^2)}\widetilde{\psi} = (X\sqrt{2}dt + \beta\sqrt{2}\circ dW_t)\sin\left(\frac{k\pi x}{H}\right)\sin\left(\frac{\pi z}{H}\right).
\label{stochasticfouriermodes}
\end{equation}
Carrying out the projection yields the Lorenz 63 system, augmented with SALT,
\begin{equation}
\begin{aligned}
{\rm \textcolor{red}{d}}X &= \sigma(Y-X)d\tau,\\
{\rm \textcolor{red}{d}}Y &= (rX-XZ-Y)d\tau - \beta Z\circ dW_\tau,\\
{\rm \textcolor{red}{d}}Z &= (XY - bZ)d\tau + \beta Y\circ dW_\tau,
\end{aligned}
\label{SALTL63}
\end{equation}
where $\sigma = \gamma\nu^{-1}$ is the Prandtl number, $r = R_aR_c^{-1}$ is a scaled Rayleigh number and $b = (4(1+k^2))^{-1}$ is a parameter related to the wavenumber. Note that there is no stochasticity in the term proportional to $r$. The equations in \eqref{SALTL63} were obtained by projecting \eqref{SOB} onto Fourier modes. In \eqref{SOB}, the stochasticity is in the nonlinear terms and not in the term that is proportional to $\alpha g$. The term proportional to $\alpha g$ is the term that corresponds to the term proportional to $r$ in \eqref{SALTL63}. The time $\tau = \pi^2(1+k^2)\gamma t H^{-2}$ is dimensionless. Hereafter, we will not distinguish between $\tau$ and $t$ and just write $t$. In \eqref{Lorenzformulation}, observe that the noise appears only in the nonlinear transport terms. We can formally rewrite \eqref{SALTL63} to show that also the associated Lorenz equations have this property,
\begin{equation}
\begin{aligned}
{\rm \textcolor{red}{d}}X &= \sigma(Y-X)dt,\\
{\rm \textcolor{red}{d}}Y &= (rX-\widetilde{X}Z-Y)dt ,\\
{\rm \textcolor{red}{d}}Z &= (\widetilde{X}Y - bZ)dt,
\end{aligned}
\label{formalSALTL63}
\end{equation}
upon introducing the notation, $\widetilde{X} = X + \beta\circ d\dot{W}_t$. The nonlinear transport terms in \eqref{SALTL63} and \eqref{formalSALTL63} represent circulation as rigid rotation of the $Y,Z$ plane around the $X$-axis at angular velocity $\widetilde{X}$. Physically, then, the  SALT modification may be interpreted as a stochastic transport around the $X$-axis with angular velocity of circulation given by $\beta \circ dW_\tau$. The It\^o form of \eqref{SALTL63} and \eqref{formalSALTL63} is
\begin{equation}
\begin{aligned}
{\rm \textcolor{red}{d}}X &= \sigma(Y-X)d\tau,\\
{\rm \textcolor{red}{d}}Y &= \left(rX-XZ-\left(1+\frac{\beta^2}{2}\right)Y\right)d\tau - \beta Z dW_\tau,\\
{\rm \textcolor{red}{d}}Z &= \left(XY - \left(b+\frac{\beta^2}{2}\right)Z\right)d\tau + \beta Y dW_\tau\,.
\end{aligned}
\label{SALTL63ito}
\end{equation}
The first equation in each of the equation sets  \eqref{SALTL63}, \eqref{formalSALTL63} and \eqref{SALTL63ito} contains no noise. This is because the Jacobian projects to zero for this Fourier expansion. A stochastic L63 system similar to that in \eqref{SALTL63ito} was also derived using the ``location uncertainty'' approach in \cite{chapron2018large}.

In \cite{chekroun2011stochastic}, a stochastic version of the Lorenz system was introduced and found to possess a pullback attractor that supports a random Sinai-Ruelle-Bowen (SRB) measure. We will not try to compute the SRB measure for our version of a stochastic Lorenz system. Instead, we will compute the numerical Lyapunov exponents for both these systems. In \cite{chekroun2011stochastic}, linear multiplicative It\^o noise is added in each variable, and the stochastic Lorenz system is formulated as
\begin{equation}
\begin{aligned}
{\rm \textcolor{red}{d}}X &= \sigma(Y-X)dt + \beta X dW_t,\\
{\rm \textcolor{red}{d}}Y &= (rX-XZ-Y)dt + \beta Y dW_t,\\
{\rm \textcolor{red}{d}}Z &= (XY - bZ)dt + \beta Z dW_t.
\end{aligned}
\label{ChekrounL63}
\end{equation}
When comparing linear terms in \eqref{ChekrounL63}, one sees that the stochasticity is paired with the dissipation. Hence we call this noise fluctuation-dissipation noise.
\section{Lyapunov Exponents}\label{LyapExp}
We now want to compare the two version of stochastic Lorenz 63 in terms of their Lyapunov exponents. Firstly, we will determine the sum of the Lyapunov exponents analytically, since this is possible for the Lorenz 63 system. To be able to compute Lyapunov exponents, a number of conditions are necessary. First of all, the system of equations needs to generate a random dynamical system (RDS), which is a tuple $(\varphi,\vartheta)$, where $\varphi$ is a cocycle, the solution of the dynamical system $\vartheta$. Additionally, an integrability condition,
\begin{equation}
\log^+\|D\varphi(t,\omega,x)\|\in L^1,
\label{METintegrabilitycondition}
\end{equation} 
where $\log^+ x:= \max(0,\log x)$, needs to be satisfied to make sure that the linear equation associated to the system that we are considering is wellposed. The integrability condition is sufficient for Oseledet's multiplicative ergodic theorem (MET), which implies the regularity and existence of Lyapunov exponents. A stochastic differential equation is locally wellposed if its solutions exist locally and are unique up to indistinguishability. Sufficient conditions are local Lipschitz continuity and a linear growth condition. The coefficients of the stochastic differential equations in \eqref{SALTL63} and \eqref{ChekrounL63} are of polynomial type, and therefore smooth. This implies local Lipschitz continuity. Both systems satisfy linear growth. After introducing the following notation, the following theorem implies that both \eqref{SALTL63} and \eqref{ChekrounL63} generate a RDS. In general, a Stratonovich stochastic differential equation can be written as
\begin{equation}
{\rm \textcolor{red}{d}}x_t = f_0(x_t)dt + \sum_{j=1}^m f_j(x_t)\circ dW_t^j = \sum_{j=0}^m f_j(x_t)\circ dW_t^j,
\end{equation}
with the convention $dW_t^0 = dt$ to allow for this shorthand. 
From \cite{arnold2003random}, we have the following theorem.
\begin{theorem}[RDS from Stratonovich SDE]
Let $f_0\in \pazocal{C}^{k,\delta}_b$, $f_1,\hdots,f_m\in\pazocal{C}^{k+1,\delta}_b$ and $\sum_{j=1}^m\sum_{i=1}^d f_j^i\frac{\partial}{\partial x_i}f_j\in\pazocal{C}^{k,\delta}_b$ for some $k\geq 1$ and $\delta > 0$. Here $\pazocal{C}^{k,\delta}_b$ is the Banach space of $\pazocal{C}^k$ vector fields on $\mathbb{R}^d$ with linear growth and bounded derivatives up to order $k$ and the $k$-th derivative is $\delta$-H\"older continuous. Then:
\begin{enumerate}[i)]
\item 
\begin{equation}
{\textcolor{red} {\sf d}}x_t = \sum_{j=0}^m f_j(x_t)\circ\d{W}_t^j, \qquad t\in\mathbb{R}
\label{stratsde}
\end{equation}
generates a unique $\pazocal{C}^k$ RDS $\varphi$ over the dynamical system (DS) describing Brownian Motion (the background theory for this can be found in \cite{arnold2003random, elworthy1978stochastic}). For any $\epsilon\in(0,\delta)$, $\varphi$ is a $\pazocal{C}^{k,\epsilon}$-semimartingale cocycle and $(t,x)\to\varphi(t,\omega)x$ belongs to $\pazocal{C}^{0,\beta;k,\epsilon}$ for all $\beta<\frac{1}{2}$ and $\epsilon<\delta$.
\item The RDS $\varphi$ has stationary independent (multiplicative) increments, i.e. for all $t_0\leq t_1\leq \hdots \leq t_n$, the random variables
\begin{align*}
\varphi(t_1)\circ\varphi(t_0)^{-1},\quad \varphi(t_2)\circ\varphi(t_1)^{-1},\quad\hdots,\quad \varphi(t_n)\circ\varphi(t_{n-1})^{-1}
\end{align*} 
are independent and the law of $\varphi(t+h)\circ\varphi(t)^{-1}$ is independent of $t$. Here $\circ$ means composition.
\item If $D\varphi(t,\omega,x)$ denotes the Jacobian of $\varphi(t,\omega)$ at $x$, then $(\varphi,D\varphi)$ is a $\pazocal{C}^{k-1}$ RDS uniquely generated by \eqref{stratsde} together with 
\begin{equation}
{\textcolor{red} {\sf d}}v_t = \sum_{j=0}^m Df_j(x_t)v_t\circ\d{W}_t^j, \qquad t\in\mathbb{R}
\label{matrixSDE}
\end{equation}
Hence $D\varphi$ uniquely solves the variational Stratonovich SDE on $\mathbb{R}$
\begin{equation}
D\varphi(t,x) = I + \sum_{j=0}^m\int_0^t Df_j(\varphi(s)x)D\varphi(s,x)\circ\d{W}^j_s, \qquad t\in\mathbb{R}
\end{equation}
and is thus a matrix cocycle over $\Theta = (\vartheta,\varphi)$.
\item The determinant $\det D\varphi(t,\omega,x)$ satisfies Liouville's equation on $\mathbb{R}$
\begin{equation}
\det D\varphi(t,x) = \exp\left(\sum_{j=0}^m\int_0^t\text{trace}(Df_j(\varphi(s)x))\circ\d{W}_s^j\right)
\label{liouville}
\end{equation}
and is thus a scalar cocycle over $\Theta$.
\end{enumerate}
\end{theorem}
The background theory and proof of this theorem can be found in \cite{arnold2003random}. The variational equation is commonly referred to as the tangent model. The matrix solutions of the tangent model are commonly called the fundamental matrix of solutions. Although both \eqref{SALTL63} and \eqref{ChekrounL63} satisfy the theorem above, an additional observation is required to guarantee that the solutions do not blow up. This is can be shown by a Lyapunov function argument, which will imply that deterministic Lorenz equations have a global attractor set. Together with the local existence and uniqueness of strong solutions to the system of SDEs, this is enough to satisfy the integrability condition. To prove the existence of a globally attracting set, one considers the Lyapunov function \cite{sparrow2012lorenz}
\begin{equation}
V(\textbf{X}) = rX^2+\sigma Y^2+\sigma(Z-2r)^2.
\end{equation}  
Dividing the total time derivative of $V(\textbf{X})$ by $2r^2\sigma b$ yields the equation for an ellipsoid
\begin{equation}
\frac{\dot{V}(\textbf{X})}{2r^2\sigma b} = -\frac{X^2}{br} - \frac{Y^2}{br} - \frac{(Z-r)^2}{r^2}+1\,.
\end{equation}
This shows that $\dot{V}$ is negative outside of the ellipsoid and positive inside the ellipsoid given by
\begin{equation}
\frac{X^2}{br}+\frac{Y^2}{br}+\frac{(Z-r)^2}{r^2}=1\,.
\end{equation}
So inside the ellipsoid the dynamics are unstable, as there is no converging behaviour. Outside of the ellipsoid, where $\dot{V}<0$, the dynamics converge towards the ellipsoid. Hence $V(\textbf{X})$ is a Lyapunov function outside of an ellipsoid. This proves that no finite time blow-up can occur for the deterministic case. Since the transport noise and fluctuation-dissipation noise Lorenz systems both satisfy linear growth, the stochastic versions also do not blow up. We are now able to prove that both \eqref{SALTL63} and \eqref{ChekrounL63} satisfy the integrability condition \eqref{METintegrabilitycondition}. For all finite systems of SDEs (so no stochastic partial differential equations), the Jacobian of the dynamics is a square matrix. Since in $\mathbb{R}^{d\times d}$ all norms are equivalent, the condition is satisfied or dissatisfied for all norms simultaneously. The Jacobian for \eqref{SALTL63} is
\begin{equation}
Df_0 + Df_1 = \left(\begin{matrix}
-\sigma & \sigma & 0\\
r - Z & -1 & - X - \beta\\
Y & X+\beta & -b
\end{matrix}\right),
\end{equation}
and for \eqref{ChekrounL63} the Jacobian is given by
\begin{equation}
Df_0 + Df_1  = \left(\begin{matrix}
-\sigma+\beta & \sigma & 0\\
r - Z & -1+\beta & - X \\
Y & X & -b + \beta
\end{matrix}\right).
\end{equation}
For Oseledet's MET to be valid we require $\|\sum_{j=0}^m Df_j\|\in L^1$. This is true if all elements of the matrices are in $L^1$. This conditions is violated if any of the elements of the matrix is unbounded, since then the argument of the logarithm would become unbounded. The Lyapunov function has proven that the dynamics have a global attractor and we have established local existence and uniqueness, so for any initial condition, the dynamics stay bounded. Therefore the integrability condition \eqref{METintegrabilitycondition} is satisfied and regularity and existence of Lyapunov exponents is guaranteed. Oseledet's MET now states that, for $v_t$ the solution of \eqref{matrixSDE}, $\lim_{t\to\infty}(v_t(\omega)^T v_t(\omega))^{1/2t} =: \Phi(\omega)\geq 0$ exists and the logarithm of the eigenvalues of $\Phi$ are the Lyapunov exponents. 
\subsection{Sum of Lyapunov exponents}
By definition of Lyapunov exponents and by using Liouville's equation \eqref{liouville} (also called Abel-Jacobi-Liouville formula), the following important fact is derived.
\begin{lemma}
If the trace of the Jacobian $Df_0$ is constant and the trace of $Df_j$ for $j\geq 1$ is zero, then the sum of the Lyapunov exponents is equal to the trace of $Df_0$. 
\end{lemma}
\begin{proof}
By taking the determinant of $\Phi$ and using several properties of the determinant for square matrices, we can show
\begin{equation}
\lim_{t\to\infty}\left(\det(v_t^T v_t)^{1/2}\right)^{1/t} = \lim_{t\to\infty}(\det v_t)^{1/t} = \lim_{t\to\infty}\left(\prod_{i=1}^n e^{\gamma_i}\right)^{1/t}.
\label{matrixidentities}
\end{equation}
First, for any square matrix $A$, $\det(A^T)=\det(A)$, which lets us write $\det(v_t^T v_t) = \det(v_t^2)$. Secondly, for any square matrix $A$ and $B$, we have $\det(AB) = \det(A)\det(B)$, which allows us to write $\det(v_t^2)=\det(v_t)^2$. Finally, the determinant of a square matrix is related to the eigenvalues of that square matrix by $\det(A)=\prod_i\kappa_i$, where $\kappa_i$ are the eigenvalues of A. Now let $e^\gamma_i$ be the eigenvalues of $v_t$. Then $\gamma_i$ are the unaveraged Lyapunov exponents. Using Liouville's equation \eqref{liouville} and the right hand side of \eqref{matrixidentities}
\begin{equation}
\lim_{t\to\infty} \exp\left(\sum_{j=0}^m\int_0^t \text{trace}(Df_j)\circ dW_s^j\right)^{1/t} = \lim_{t\to\infty}(\det v_t)^{1/t} = \lim_{t\to\infty}\left(\prod_{i=1}^n e^{\gamma_i}\right)^{1/t} = \lim_{t\to\infty}\exp\left(\sum_{i=1}^n\gamma_i\right)^{1/t}
\end{equation}
Since the trace of the Jacobian $Df_0$ was assumed to be constant and the trace of Jacobian $Df_j$ is zero for all $j\geq 1$, by taking the logarithm, we find
\begin{equation}
\sum_{i=1}^n \lambda_i = \lim_{t\to\infty}\big(\text{trace}(Df_0)\big)\frac{t}{t}= \text{trace}(Df_0),
\end{equation}
where $\lambda_i$ are the Lyapunov exponents by definition.
\end{proof}

This lemma applies to the Lorenz 63 system with SALT \eqref{SALTL63}, where it implies that the sum of the Lyapunov exponents is equal to that of the deterministic system
\begin{equation}
\sum_{i=1}^3\lambda_i = -\sigma - 1 -b.
\end{equation} 
This conclusion does not hold for the Lorenz 63 system with fluctuation dissipation noise \eqref{ChekrounL63}, because for the latter the trace of $Df_1$ is nonzero. Still one can use the Liouville equation \eqref{liouville} and a similar computation to that in the proof of the lemma, to obtain the sum of the Lyapunov exponents for \eqref{ChekrounL63}
\begin{equation}
\sum_{i=1}^3\lambda_i = -\sigma -1 -b + 3\beta\lim_{t\to\infty}\frac{W_t}{t}\,.
\end{equation}
The sum of Lyapunov exponents represents the average rate of expansion or contraction of phase-space volume. Hence this result shows on a theoretical level that the phase-space contraction (or expansion) properties of the two systems are different for any finite time. In the limit, the two systems have the same properties.
\section{Method to compute Lyapunov exponents}
To compute the finite time approximation to the Lyapunov exponents (which we will refer to as numerical Lyapunov exponents, or NLEs), one needs to simultaneously solve the governing dynamics and the corresponding variational equation. When the dynamics are given by a system of differential equations, the variational equation becomes a matrix differential equation. The appropriate initial condition for the variational equation is usually the identity matrix, as this corresponds to evolving the unit ball along the linearised dynamics. The ball deforms and its average deformation is associated to the NLEs. However, directly solving the variational equation is problematic, as the vectors associated to the NLEs tend to align along the direction of largest increase. Regularly orthonormalising avoids this issue. The well known QR method is therefore the go-to option for solving the variational equation. The QR method dictates that the matrix $v_t$ is decomposed into an orthogonal matrix $Q\in O(n):= \{ Q\in\mathbb{R}^{n\times n}: \det Q = \pm 1\}$ and an upper triangular matrix $R$ at every time step. Consider the Stratonovich SDE on $\mathbb{R}^n$ given by
\begin{equation}
{\sf \textcolor{red} d}Y_t = \sum_{j=0}^m f_j(Y_t)\circ dW_t^j,
\end{equation} 
where the functions $f_j$ are sufficiently regular to guarantee wellposedness of the SDE. Again the convention $dW_t^0 = dt$ is used. The corresponding variational equation is
\begin{equation}
{\sf \textcolor{red} d}v_t = \sum_{j=0}^m J_j v_t\circ dW_t^j, \qquad v_0 = I, \quad v_t\in\mathbb{R}^{n\times n},
\end{equation}
where $J_j := Df_j(Y_t)$ is the Jacobian of the dynamical system and $I\in\mathbb{R}^{n\times n}$ is the identity matrix. Applying the QR decomposition, multiplying by $Q^T$ from the left and by $R^{-1}$ from the right then yields
\begin{equation}
Q^T{\sf \textcolor{red} d}Q + {\sf \textcolor{red} d}RR^{-1} = \sum_{j=0}^m Q^T J_j Q\circ dW_t^j.
\label{QRvariational}
\end{equation}
The first matrix on the left hand side is skew-symmetric and the second matrix is upper triangular. This fact will be used in solving for $Q$ and $R$ independently. It is necessary to take a new $QR$-decomposition once every so often, otherwise the matrix $Q$ may lose its orthogonality and cause the algorithm to break down. In addition, for high-dimensional dynamical systems, by means of the Cayley transform, in \cite{udwadia2002efficient} the $QR$-method is slightly adapted to gain a small speed-up.  The Cayley transform will provide an orthogonal matrix with determinant +1 as long as the eigenvalues do not approach -1. This may occur during the time integration, so a restarting procedure can avoid this potential problem. The restarting procedure is possible due to the following lemma.

In \eqref{matrixSDE}, for $t>t_0$ set $v_{t_0} = Q_0R_0$, where $Q_0$ is orthogonal and $R_0$ is upper triangular with all diagonal elements positive. As in \cite{udwadia2002efficient}, the real line is divided into subintervals $t_i\leq t \leq t_{i+t}$ for $i=1,2,\hdots,$ so that each interval has length $\Delta t_i = t_{i+1} - t_i$. The solution $v_{t_i}$ to the variational equation \eqref{matrixSDE} at time $t_i$ can be decomposed as $v_{t_i} = Q_i R_i$ for $i=1,2,\hdots$. We can now introduce the following lemma.
\begin{lemma}
At any time $t=t_i + \tau$, $0\leq \tau \leq \Delta t_i$, for $i=1,2,\hdots$, the solution of the variational equation \eqref{matrixSDE} can be expressed as
\begin{equation}
v_t = v(t_i +\tau) = Q_i  \widetilde{v}_\tau R_i = Q_i\widetilde{Q}_\tau \widetilde{R}_\tau R_i, \qquad 0\leq \tau \leq \Delta t_i,\quad t_i\leq t\leq t_{i+1},
\end{equation}
where $\widetilde{v}_\tau$ is the solution to the differential equation given by
\begin{equation}
{\sf \textcolor{red} d}\widetilde{v}_\tau = \sum_{j=0}^m \widetilde{J}_j(\tau)\widetilde{v}_\tau\circ dW_\tau^j,\qquad 0\leq\tau\leq \Delta t_i, \quad \widetilde{v}_0 = I,\quad i = 0,1,2,\hdots
\end{equation}
with $Q_0 =I$, $R_0=I$ and $\widetilde{J}_j(\tau) = Q_i^T J_j(t_i + \tau)Q_i$.
\end{lemma}   
The proof of this lemma can be found in \cite{udwadia2002efficient}. Although in that proof the variational equation is deterministic, the stochastic case is straightforwardly found from the deterministic one, as the only change is the variational equation itself. The Cayley transform is defined as
\begin{equation}
Q = (I-K)(I+K)^{-1},
\label{Cayleytransform}
\end{equation}
where $I\in\mathbb{R}^{n\times n}$ is the identity matrix and $K\in\mathbb{R}^{n\times n}$ is a skew-symmetric matrix. An important feature of $(I-K)$ and $(I+K)^{-1}$ is that they commute. The transformation \eqref{Cayleytransform} is valid as long as none of the eigenvalues of $Q$ are equal to -1. By deriving the stochastic differential equation for $K$, the Cayley method takes form. Since the initial condition for $Q(0) = I$, the initial condition for $K(0) = 0$. By taking the stochastic evolution differential of $Q$, we obtain
\begin{equation}
{\sf \textcolor{red} d}Q = -{\sf \textcolor{red} d}K(I+K)^{-1} - (I-K)(I+K)^{-1}{\sf \textcolor{red} d}K(I+K)^{-1}.
\end{equation}
Using the skew-symmetry of $K$, the distributive property of the transpose, the fact that for any invertible matrix $A$ we have that $(A^T)^{-1} = (A^{-1})^T$ and by writing $(I+K) = -(I-K)+2I$ and setting $H:=(I+K)^{-1}$, one obtains the following expression
\begin{equation}
\begin{aligned}
Q^T{\sf \textcolor{red} d}Q &= -2(I-K)^{-1}{\sf \textcolor{red} d}K(I+K)^{-1}\\
&= -2H^T{\sf \textcolor{red} d}K H.
\end{aligned}
\label{CayleyLHS}
\end{equation}
Upon introducing the notation $G:=(I-K)$ and using $H$ as before, the right hand side of \eqref{QRvariational} is expressed as
\begin{equation}
\sum_{j=0}^m Q^T J_j Q\circ dW_t^j = \sum_{j=0}^m H^T G^T J_j GH\circ dW_t^j .
\label{CayleyRHS}
\end{equation}
Substitution of \eqref{CayleyLHS} and \eqref{CayleyRHS} into \eqref{QRvariational} yields the following differential equation
\begin{equation}
-2H^T{\sf \textcolor{red} d}KH + {\sf \textcolor{red} d}R R^{-1} = \sum_{j=0}^m H^T G^T J_j GH\circ dW_t^j.
\label{CayleySDE}
\end{equation}
Recall that the first matrix on the left hand side is skew-symmetric and the second matrix is upper triangular. Let $S = H^T{\sf \textcolor{red} d}KH$ so that
\begin{equation}
\begin{aligned}
S_{ab} = \begin{cases}
\frac{1}{2}\left(\sum_{j=0}^m H^TG^T J_j GH\right)_{ab}\circ dW_t^j &\qquad \text{for} \quad a>b\\
\qquad \qquad 0 &\qquad \text{for} \quad a = b\\
-\frac{1}{2}\left(\sum_{j=0}^m H^TG^T J_j GH\right)_{ab}\circ dW_t^j &\qquad \text{for} \quad a<b\,,
\end{cases}
\end{aligned}
\end{equation}
which determines the differential equation for $K$ as
\begin{equation}
\begin{aligned}
{\sf \textcolor{red} d}K = H^{-T}SH^{-1} = \begin{cases}
(G^TSG)_{ab} &\qquad \text{for} \quad a > b\\
\qquad 0 & \qquad \text{for} \quad a=b\\
-(G^TSG)_{ab} &\qquad \text{for} \quad a<b\,.
\end{cases}
\end{aligned}
\end{equation}
Observe that since $K$ is skew-symmetric, it is determined by the lower triangular part of $G^TSG$. Now that $K$ is known, the Lyapunov exponents are determined as the averages of the solutions of the differential equation for $\rho_a := \log(R_{aa})$,
\begin{equation}
{\sf \textcolor{red} d}\rho_a = \sum_{j=0}^m h_a^TG^T J_j G h_a \circ dW_t^j, \qquad \rho_a(0)=0,
\label{LyapunovExponentsCayley}
\end{equation}
where $h_a$ are the columns of $H = [h_1 \quad h_2 \quad \hdots \quad h_n]$. The Lyapunov exponents are then found as
\begin{equation}
\lambda_a = \lim_{t\to\infty}\frac{\rho_a}{t}.
\end{equation}
Recall that this solution method is valid as long as the eigenvalues of $Q$ do not equal -1. As the initial condition is $Q(0)=I$, there is always an interval of time $0\leq t\leq t_0$ in which the condition for the Cayley transform \eqref{Cayleytransform} is not violated. The following condition for restarting the algorithm is introduced: let $\eta\in(0,1)$ be chosen by the user such that $\|K\|\leq \eta < 1$ for some suitable norm. At time $t_0$, when the norm of $K$ equals $\eta$, $Q(t_0) =: Q_0$ is computed and stored. The algorithm is then restarted at time $t_0$, where due to the lemma we have
\begin{equation}
{\sf \textcolor{red} d}v_t = \sum_{j=0}^m Q_0^TJ_j Q_0 v_\tau \circ dW_\tau^j = \sum_{j=0}^m \widetilde{J}_jv_\tau \circ dW_\tau^j
\end{equation}
for $t_0 \leq \tau$. Besides the adapted Jacobian, this is the same equation and hence can be solved using the Cayley method as long as the norm of $K$ is smaller than our chosen value for $\eta$. The initial condition for equation \eqref{LyapunovExponentsCayley} changes to $\rho(0) = \rho(t_0)$.
\section{Numerical Results}
The Lorenz system has been studied intensively with the standard parameter values $\sigma = 10$, $r = 28$ and $b=8/3$, \cite{lorenz1963deterministic, arnold2001lyapunov}, though in the latter two papers for an adapted version of the Lorenz system. Lorenz showed that for these parameter values, the deterministic Lorenz has a strange attractor. In \cite{wolf1985determining}, the Lorenz system is studied for the nonstandard parameter values $\sigma = 16$, $r=45.92$ and $b=4$.  Upon introducing stochasticity, one can no longer speak of an attractor in the standard sense, since the noise, due to the unbounded variation nature of the Wiener process, will push the dynamics out of any bounded set almost surely. We set the initial condition to $(X(0),Y(0),Z(0)) = (0,1,0)$ and evolve the system for 50000 time steps with time step size $\Delta t = 0.001$ with the Euler-Maruyama method. This sets the initial condition for determining the numerical Lyapunov exponents. We solve the SDEs in the Cayley method with the same time step $\Delta t$ for $10^5$ iterations in total. The norm tolerance for the matrix $K$ is set to $\eta = 0.8$. It is known that the Euler-Maruyama method has poor convergence, so the individual exponents can be computed much more accurately if one were to implement a better numerical method. For the deterministic case, the individual exponents agree reasonably well with the existing literature. The sum of the numerical Lyapunov exponents turns out to be a very robust value, as even the Euler-Maruyama method establishes the correct value to high accuracy. 
\bigskip

\subsection{Deterministic Case}
When there is no noise, the Liouville equation implies that sum of the Lyapunov exponents is equal to the trace of the Jacobian of the dynamics. For the standard parameter values $\sigma = 10, r = 28$ and $b = 8/3$, the sum of the Lyapunov exponents is
\begin{equation}
\sum_{i=1}^3 \lambda_i = -\sigma - 1 -b \approx - 13.6667,
\end{equation}
The individual NLEs for the deterministic Lorenz 63 system have been determined by \cite{sprott2003chaos} who used a 4th order Runge-Kutta method with a fixed step size of 0.001, performed over $10^9$ iterations. The Cayley method allows us to determine the individual NLEs, where we solve the dynamics with a forward Euler scheme. We find the following values
\begin{table}[H]
\centering
\begin{tabular}{lccc|c}
& $\lambda_1$ & $\lambda_2$ & $\lambda_3$ & $\sum_{i=1}^3\lambda_1$\\
\hline & & & &\\
Cayley method with forward Euler & 0.8739 & -0.0798 & -14.4604 & -13.6665\\
& & & &\\
Values according to \cite{sprott2003chaos,sparrow2012lorenz} & 0.9056 & 0 & -14.5721 & -13.6665
\end{tabular}
\caption{The individual NLEs and sum for $\sigma = 10$, $r=28$ and $b=8/3$ as computed with the Cayley method and those found in literature.}
\end{table}
The values are not exactly the same for the individual NLEs. This is due to the poor convergence of the numerical methods used here (Euler-Maruyama has order 1/2 convergence). However, the sum is the same in four decimal places. As an additional comparison, we also study the deterministic system for the nonstandard parameter values used in \cite{wolf1985determining},  $\sigma = 16, r = 45.92$ and $b=4$, where the sum is 
\begin{equation}
\sum_{i=1}^3\lambda_i =-\sigma-1-b= -21.
\end{equation}
In this situation, we find the following values
\begin{table}[H]
\centering
\begin{tabular}{lccc|c}
& $\lambda_1$ & $\lambda_2$ & $\lambda_3$ & $\sum_{i=1}^3\lambda_1$\\
\hline & & & &\\
Cayley method with forward Euler & 1.4858 & -0.0721 & -22.4135 & -20.9998\\
& & & &\\
Values according to \cite{wolf1985determining} & 1.50 & 0 & -22.46 & -22.96
\end{tabular}
\caption{The individual NLEs and sum for $\sigma = 16$, $r=45.92$ and $b=4$ as computed with the Cayley method and those found in literature.}
\end{table}
Again the values are not the same for the individual NLEs, but the sum is accurate. For both sets of parameter values, the individual NLEs are computed in good agreement with those found in literature.
\subsection{Stochastic Cases}
Here the Lorenz system is studied with SALT. The noise amplitude is chosen to be $\beta = 0.5$. In figure \ref{saltl63figure} it can be seen that the stochastic dynamics give rise to a perturbed butterfly shaped object in phase space. The numerical Lyapunov exponents converge. The $x$-axis in the convergence plot of the NLEs shows time, which is simply the number of iterations multiplied by the time step. The sum of the individual NLEs evaluates to -13.6665, which is the value produced by the deterministic algorithm as well. This agrees with the analysis done in section \ref{LyapExp}.
The individual NLEs for both stochastic versions of the Lorenz 63 system behave very similarly when the noise amplitude is increased. In both cases, the top two exponents decrease when the noise amplitude increases, whereas the bottom exponent increases. For the system with SALT, the bottom exponent increases faster than in the fluctuation-dissipation system. Thus the sum of the NLEs is independent of the noise, while the individual NLEs are not. Analytical results for the individual NLEs are difficult to establish and often only estimates are possible. We therefore restrict ourselves here to studying the value of the sum of the exponents, rather than the individual NLEs.
\begin{figure}[H]
\centering
\includegraphics[width = \textwidth]{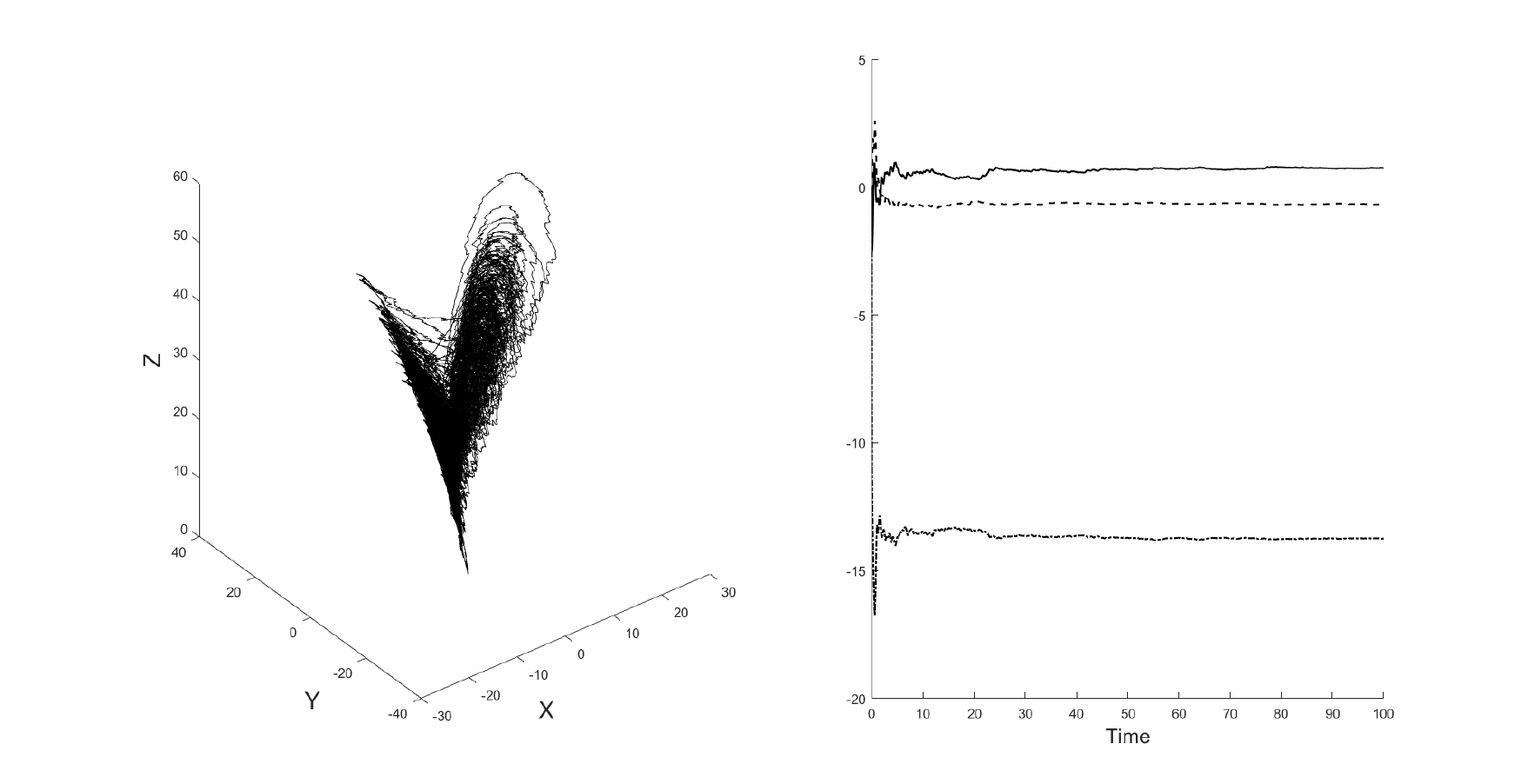}
\caption{The Lorenz system with SALT with noise amplitude $\beta = 0.5$. The left figure shows a single realisation of the stochastic dynamics. The right figure shows the convergence of the individual numerical Lyapunov exponents.}
\label{saltl63figure}
\end{figure}
The Lorenz system with fluctuation-dissipation noise, with noise amplitude $\beta = 0.5$, gives rise to the plots shown in figure \ref{fdl63figure}. The NLEs converge and their sum for this particular realisation is -13.7636. This is a discrepancy in the first decimal place compared to both the deterministic and the SALT case, which agree.
\begin{figure}[H]
\centering
\includegraphics[width = \textwidth]{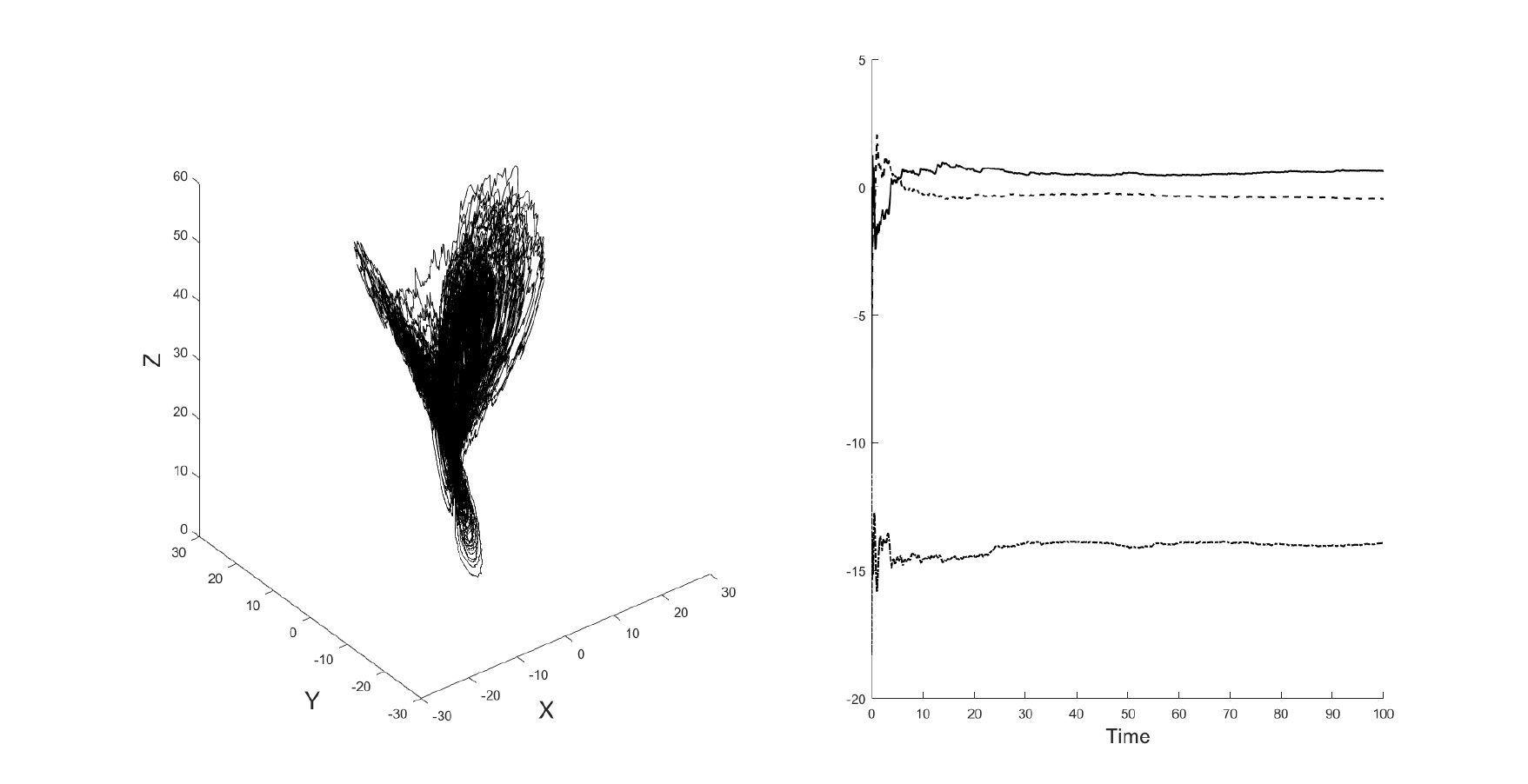}
\caption{The Lorenz system with fluctuation-dissipation noise with noise amplitude $\beta = 0.5$. The left figure shows a single realisation of the stochastic dynamics. The right figure shows the convergence of the individual numerical Lyapunov exponents.}
\label{fdl63figure}
\end{figure}
From figures \ref{saltl63figure} and \ref{fdl63figure} it can be seen in the left panels that both systems give rise to a locus of phase space trajectories that vaguely resemble a butterfly. The realisation of the Wiener process is the same for both figures. It can be seen that the two types of noise give rise to different trajectories in phase space. The right panels show that for both systems the NLEs have converged. The constancy of the sum for SALT becomes especially clear in figure \ref{sumplot}, which shows a hundred computations for increasing noise amplitude. Each time the noise amplitude is increased, a new realisation of the Wiener process is used. The solid black line is the sum of the NLEs for fluctuation-dissipation noise, which depends on each realisation, whereas the dashed line, which shows the sum of the NLEs for SALT, is constant at -13.6665 and independent of noise amplitude and realisation of the Wiener process. For a fixed realisation of the Wiener process, the theory predicts a linear behaviour. This can be seen in figure \ref{sumplotlinear}.
\begin{figure}[H]
\centering
\includegraphics[angle = 270, width = .8\textwidth]{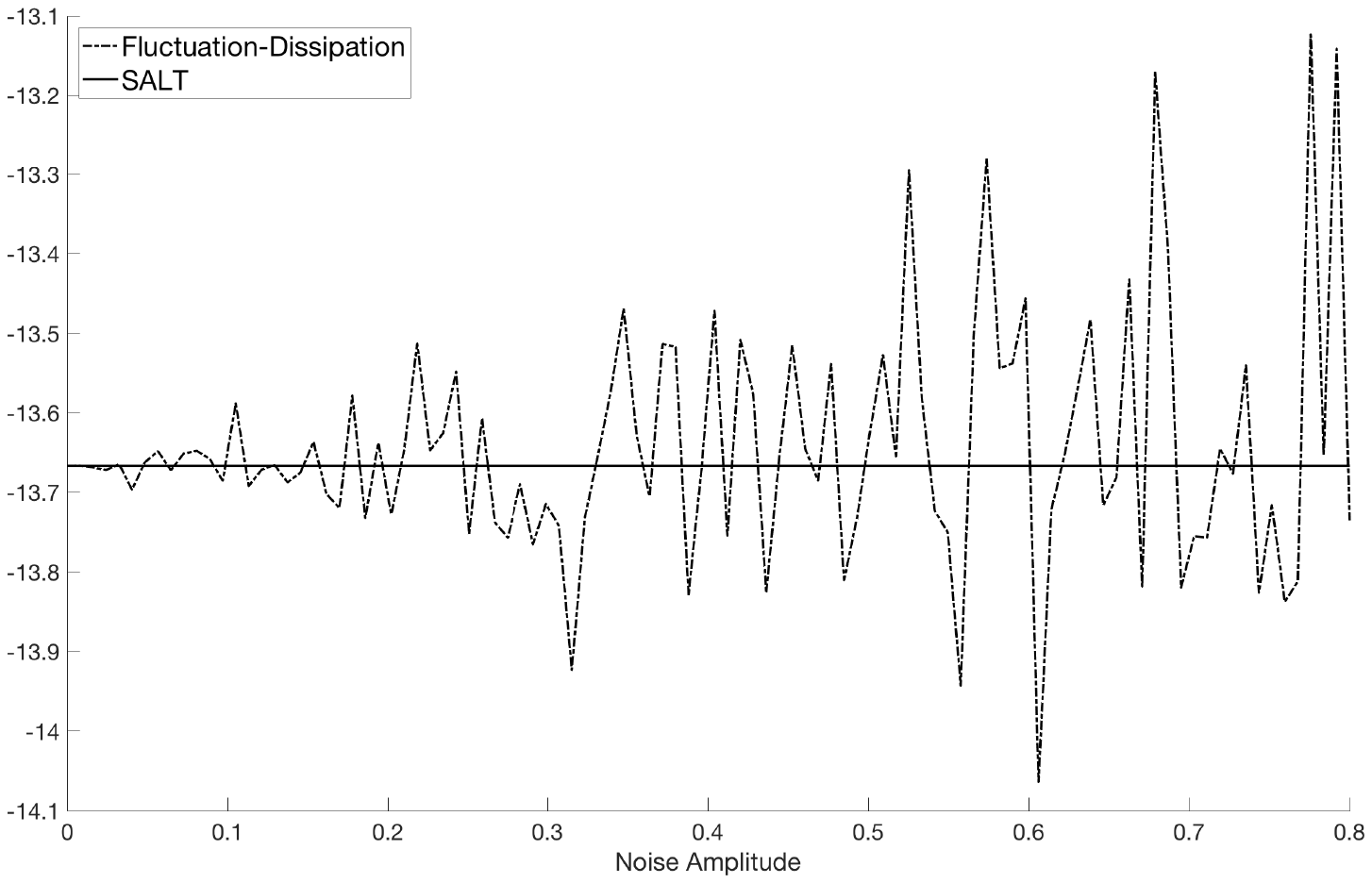}
\caption{The sum of NLEs for the two different types of stochasticity. Each increase in noise amplitude corresponds to a new realisation of the Wiener process. The sum of the NLEs for SALT is constant and for fluctuation-dissipation noise it varies for each realisation.}
\label{sumplot}
\end{figure}
\begin{figure}[H]
\centering
\includegraphics[angle = 270, width = .8\textwidth]{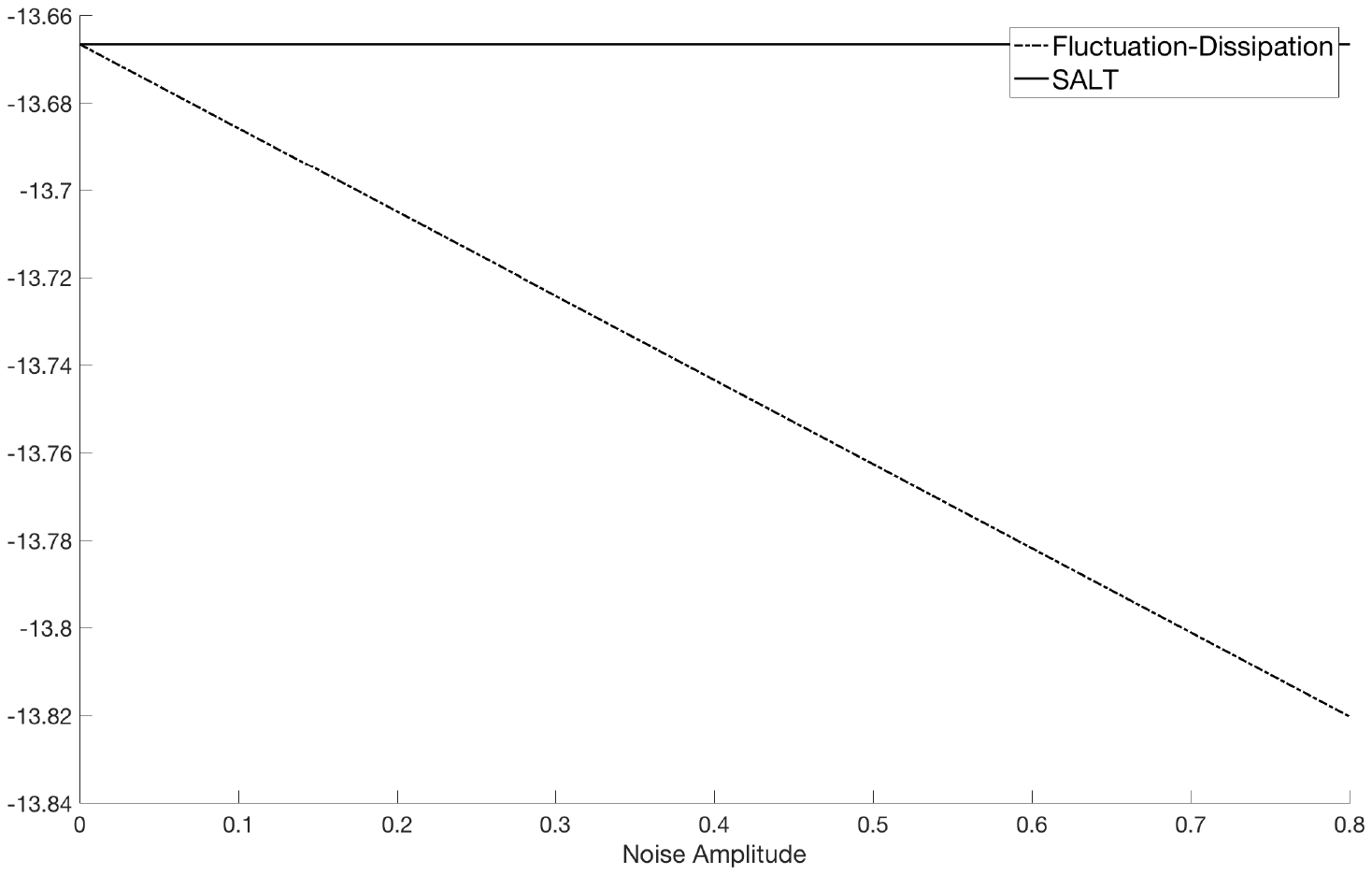}
\caption{The sum of NLEs for the two different types of stochasticity for increasing noise amplitude. The realisation of the Wiener process is fixed throughout the computation. The sum of NLEs for SALT is constant and for fluctuation-dissipation noise it decreases linearly with increasing noise amplitude.}
\label{sumplotlinear}
\end{figure} 
The individual Lyapunov exponents for systems \eqref{SALTL63} and \eqref{ChekrounL63} are almost identical for each realisation, although the sums are different, since the phase-space volume contraction rate is preserved for the system \eqref{SALTL63}.

\section{Conclusion}\label{sec: conclude}
We used the Kelvin circulation theorem for the ideal Oberbeck-Boussinesq equations to modify the equation of motion and the advection equation for the heat so as to include stochastic advection by Lie transport (SALT). We then added viscosity in the motion equation and heat diffusivity in the advection equation to obtain the Oberbeck-Boussinesq equations with SALT. By using a specific truncated Fourier series expansion of these equations, the corresponding Lorenz 63 system with SALT was obtained in equation set \eqref{SALTL63ito}. This low-dimensional system of stochastic differential equations was then compared with a Lorenz 63 system \eqref{ChekrounL63}, which was perturbed using fluctuation-dissipation noise (linear multiplicative noise in each variable), which has been previously studied in \cite{chekroun2011stochastic}. 
\bigskip

By applying methods from the theory of random dynamical systems, we were able to show that the two types of systems \eqref{SALTL63ito} and \eqref{ChekrounL63} have rather different qualitative properties. In particular the linear multiplicative noise in each variable of the fluctuation-dissipation system changes the average rate of contraction or expansion of phase space volume, relative to the deterministic system; whereas the system with SALT conserves this rate. This means that the type of noise introduced into low-dimensional dynamical systems can affect properties of the underlying deterministic system, so one should examine the effects of stochasticity on a qualitative level. For example, when a system of deterministic equations is Hamiltonian, introducing arbitrary noise would destroy the Hamiltonian structure by altering the average rate of phase space volume contraction, which for Hamiltonian systems is exactly zero. SALT would preserve this property. See \cite{holm2015variational, arnaudon2015noise, holm2016variational, holm2017uncertainty, cotter2017stochastic, arnaudon2018noise, arnaudon2018stochastic, cruzeiro2018momentum, holm2018stochastic, crisan2019solution}.
\bigskip

For the numerical verification of our analytical results for the stochastic Lorenz systems, we introduced a stochastic generalisation of the Cayley method for the computation of numerical Lyapunov exponents (NLEs). This method is a QR-based algorithm in which the orthogonal matrix is determined via the Cayley transform. The method turned out to be robust and stable for the stochastic case. Improvements to the numerical calculations are possible by using better numerical methods. However, the numerical results we obtained here agree with our analytical results and the individual NLEs for the deterministic Lorenz system agree reasonably well with those found in literature. In the stochastically perturbed Lorenz systems, the numerical method is able to distinguish clearly between the two types of noise.

\section{Outlook for a new climate science paradigm}\label{sec: outlook}
In this paper we have investigated the effect of Stochastic Advection by Lie Tranport (SALT) on Rayleigh-B\'enard convection and how the underlying Lorenz 63 system changes when the SALT approach is applied. In a celebrated unpublished paper \cite{lorenz1995climate} Ed Lorenz defined the statistical approach to climate science by the adage, ``Climate is what you expect, weather is what you get." That is, climate science is fundamentally probabilistic. This realisation motivates the use of so-called Lagrangian Averaged SALT (LA-SALT) approach, first introduced in \cite{drivas2018circulation} and developed further in \cite{drivas2019lagrangian} and \cite{alonso2019modelling}. Briefly put, the LA-SALT approach alters the stochastic material velocity $\widehat{\textbf{u}}$ of the circulation loop in Kelvin's theorem for SALT in equation \eqref{SALTkelvinthm} by replacing the drift velocity $\textbf{u}$ by its expected value $\mathbb{E}[\textbf{u}]$. Namely,
\begin{equation}
\widehat{\textbf{u}}\rightarrow {\sf \textcolor{red} d}\widehat{\mathbf{y}_t} = \mathbb{E}[\textbf{u}]\,dt + \sum_i {\boldsymbol \xi}_i\circ dW_t^i\,,
\label{lasaltkelvin}
\end{equation}
where $\mathbb{E}[\,\cdot\,]$ denotes expectation. 
The altered velocity of the circulation loop in \eqref{lasaltkelvin} allows one to take the expectation of the entire Kelvin circulation theorem. Consequently, the equations governing the expected solution comprise a deterministic subsystem of the stochastic Kelvin circulation theorem for the stochastic velocity vector field in \eqref{lasaltkelvin}. By following the approach illustrated in the present paper from equation \eqref{SOB} onward in the present paper, one may derive the LA-SALT version of the Lorenz 63 system. Perhaps not surprisingly, after replacing the transport drift velocity $X$ in the nonlinear terms in equations \eqref{formalSALTL63} by its expectation $\mathbb{E}[X]$ these equations become
\begin{equation}
\begin{aligned}
{\rm \textcolor{red}{d}}X &= \sigma(Y-X)d\tau,\\
{\rm \textcolor{red}{d}}Y &= \left(rX-\mathbb{E}[X]Z-\left(1+\frac{\beta^2}{2}\right)Y\right)d\tau - \beta Z dW_\tau,\\
{\rm \textcolor{red}{d}}Z &= \left(\mathbb{E}[X]Y - \left(b+\frac{\beta^2}{2}\right)Z\right)d\tau + \beta Y dW_\tau.
\end{aligned}
\label{SALTL63lasalt}
\end{equation}
Taking the expectation of each of the equations in \eqref{SALTL63lasalt} then yields a deterministic dynamical system for the expected solutions, 
\begin{equation}
\begin{aligned}
\frac{d}{d\tau}\mathbb{E}[X] &= \sigma\big(\mathbb{E}[Y]-\mathbb{E}[X]\big)\,,\\
\frac{d}{d\tau}\mathbb{E}[Y] &= r\mathbb{E}[X]-\mathbb{E}[X]\mathbb{E}[Z]-\left(1+\frac{\beta^2}{2}\right)\mathbb{E}[Y]\,,\\
\frac{d}{d\tau}\mathbb{E}[Z] &= \mathbb{E}[X]\mathbb{E}[Y] - \left(b+\frac{\beta^2}{2}\right)\mathbb{E}[Z]
\,.
\end{aligned}
\label{SALTL63lasalt2}
\end{equation}
The system \eqref{SALTL63lasalt2} is isomorphic to the deterministic Lorenz 63 system with a slightly renormalised dissipation. Hence this system has a strange attractor and one can make the following analogy to Lorenz's quote. The Lorenz 63 system with SALT in \eqref{SALTL63} is like the weather. (It's what you get.) and the Lorenz 63 system with LA-SALT in \eqref{SALTL63lasalt2} is like the climate. (It's what you expect.) Since the LA-SALT approach recovers the deterministic Lorenz 63 equations as a subsystem, all of the dynamical systems analysis which has been done for the deterministic Lorenz 63 system previously now carries over to the  LA-SALT expectation equations in \eqref{SALTL63lasalt2}. 
\bigskip

Finally, if one subtracts the deterministic expected LA-SALT equations in \eqref{SALTL63lasalt2} from the stochastic  LA-SALT equations in \eqref{SALTL63lasalt}, one finds a system of linear stochastic equations driven by the expected solutions of the system \eqref{SALTL63lasalt2} for the differences, 
\begin{equation}
X' = X - \mathbb{E}[X]\,,\quad Y' = Y - \mathbb{E}[Y]\,,\quad Z' = Z - \mathbb{E}[Z]\,,
\label{fluct-defs}
\end{equation}
regarded as stochastic fluctuations away from the Lorenz 63 system. Namely, 
\begin{equation}
\begin{aligned}
{\rm \textcolor{red}{d}}X' &= \sigma(Y'-X')d\tau,\\
{\rm \textcolor{red}{d}}Y' &= \left(rX'-\mathbb{E}[X]Z'-\left(1+\frac{\beta^2}{2}\right)Y'\right)d\tau 
- \beta (Z' +  \mathbb{E}[Z] )dW_\tau,\\
{\rm \textcolor{red}{d}}Z' &= \left(\mathbb{E}[X]Y' - \left(b+\frac{\beta^2}{2}\right)Z'\right)d\tau 
+ \beta \big(Y' +  \mathbb{E}[Y] \big)dW_\tau.
\end{aligned}
\label{SALTL63lasalt-fluct}
\end{equation}

We would like to dedicate future work to the further investigation of the systems in equations \eqref{SALTL63lasalt}--\eqref{SALTL63lasalt-fluct} and their implications for the differences between `what you expect' in \eqref{SALTL63lasalt2} and the dynamical evolution of the statistics of `what you get' in \eqref{SALTL63lasalt-fluct}. In particular, the accuracy of the prediction of `what you expect' in the chaotic solutions on the strange attractor of the equation \eqref{SALTL63lasalt2} depends on the amplitude of the noise representing `what you don't know' in \eqref{lasaltkelvin}. Moreover, the time-dependent variances $\mathbb{E}[(X')^2]$, $\mathbb{E}[(Y')^2]$ and $\mathbb{E}[(Z')^2]$ of `what you get' derived from the equation set \eqref{SALTL63lasalt-fluct} depend on the amplitude of the noise in \eqref{lasaltkelvin} and they are driven by the chaotic evolution of `what you expect' in \eqref{SALTL63lasalt2}. Even with the simplifications that (i) the expectations of the solutions decouple as a closed nonlinear chaotic system isomorphic to the Lorenz 63 system in \eqref{Lorenz63}, and (ii) the stochastic evolution of the fluctuations in \eqref{fluct-defs} is linear, the dynamics of the statistics of the fluctuation will be intricate and challenging to characterise fully. 

\section*{Acknowledgements}
We are enormously grateful for many suggestions for improvements offered in encouraging discussions with M. Chekroun, J. de Cloet, D.O. Crisan, M. Engel, V. Lucarini, J.C. McWilliams, E. M\'emin, M. Rasmussen, V. Resseguier and   S. Takao. The work of DDH was partially supported by EPSRC Standard Grant [grant number EP/N023781/1] and EL was supported by [grant number EP/L016613/1]. EL is grateful for the warm hospitality at the Imperial College London EPSRC Centre for Doctoral Training in Mathematics of Planet Earth (\url{mpecdt.org}) during the course of this work. 
 
\bibliographystyle{alpha}
\bibliography{bibliography}
\end{document}